\documentclass[11pt,a4paper]{article}

\usepackage[utf8]{inputenc}
\usepackage{amsfonts}
\usepackage{amssymb}
\usepackage{amsmath}
\usepackage{amsthm}
\usepackage[all]{xy}
\usepackage{enumerate}
\usepackage{color}
\usepackage{graphicx}
\usepackage{tkz-base}
\usepackage{titlesec}
\usepackage{hyperref}
\hypersetup{colorlinks=true,linkcolor=blue,filecolor=bluegray,citecolor = darkgray,urlcolor=blue}

\newcommand{\N}{\mathbb{N}}
\newcommand{\R}{\mathbb{R}}
\newcommand{\Id}{\operatorname{Id}}

\newcommand{\Morph}{\operatorname{Morph}}

\begin{document}
\theoremstyle{plain}
\newtheoremstyle{myexstyle}
{}{}{}{}{\bfseries}{.}{ }{}
\newtheorem{theorem}{Theorem}[section]
\newtheorem{corollary}[theorem]{Corollary}
\newtheorem{definition}[theorem]{Definition}
\newtheorem{lemma}[theorem]{Lemma}
\newtheorem{notation}[theorem]{Notation}
\newtheorem{convention}[theorem]{Convention}
\newtheorem{proposition}[theorem]{Proposition}
\theoremstyle{myexstyle}
\newtheorem{remark}[theorem]{Remark}
\newtheorem{example}[theorem]{Example}
\newtheorem{examples}[theorem]{Examples}
\numberwithin{equation}{section}

\title{Structure of finite dimensional linear bundles morphisms}

\author{Fernand PELLETIER $^1$ and Patrick CABAU $^2$}
\date{}

\maketitle

\tableofcontents

\begin{abstract}
Let $p_E : E \to M$ be a fibre bundle over the $m$-dimensional manifold $M$ whose typical fibre is the vector space $\R^e$ and let $p_F : F \to N$ be a fibre bundle over the $n$-dimensional manifold $N$ whose typical fibre is the vector space $\R^f$.\\
We are interested in the structure of the set 
$\Morph(E,F)$ of smooth linear bundle morphisms $\Phi : E \to F$ above smooth maps $\varphi : M \to N$.
\end{abstract}


\medskip
\noindent
\textbf{Acknowledgments}.- The authors would like to thank Pr Eduardo Mart\'inez (University of Zaragoza, Spain) for having submitted this problem.\\

\bigskip
{\footnotesize
\noindent  
$^1$ Univ. Savoie Mont Blanc, CNRS, LAMA, 73000 Chamb\'{e}ry, France\\
fernand.pelletier@univ-smb.fr
\smallskip \\
$^2$ Univ. Savoie Mont Blanc, CNRS, LAMA, 73000 Chamb\'{e}ry, France\\
patrickcabau@yahoo.fr}
\medskip

\section{Introduction} This note is an answer to a question asked by Pr Eduardo Mart\'inez (University of Zaragoza, Spain) about the type of structure the set of  vector bundle morphisms between finite dimensional vector bundles over finite dimensional manifolds could be endowed. In this purpose,  according to the book on the  convenient setting  (cf. \cite{KrMi97}), given finite dimensional vector bundles $p_E:E\to M$ and $p_F:F\to N$, we show  that the set $ \Morph(E,F)$ of vector bundle morphisms  $\Phi: E\to F$ over smooth maps $\varphi:M\to N$ can be provided with a structure of  convenient vector bundle $\pi:\Morph(E,F)\to C^\infty(M,N)$  over the convenient manifold of smooth maps $C^\infty(M,N)$.

\section{Preliminaries and notations}
In all this note, $M$ (resp. $N$) is a second countable manifold of dimension $m$ (resp. $n$).\\
\subsection{Trivializations and transition functions of a linear bundle}
\label{_TrivializationsAndTransitionFunctionsOfALinearBundle}

 Since $F$ is a vector bunlde, there exists a covering $\{V_\alpha \}_{\alpha\in A}$\footnote{Where $A$ is a countable set.} of $N$  by open connected sets $V_\alpha$ in $N$ such that,  for each $\alpha\in A$, $F_{\vert V_\alpha}=p_F^{-1}(V_\alpha)$ is trivializable. 
Therefore  there exists a diffeomorphism (trivialization) 
$\tau_\alpha: p_F^{-1}(V_\alpha) \to V_\alpha \times \R^f$ satisfying the following conditions:
\begin{enumerate}
\item[\textbf{(TFB1)}]
$\mathrm{pr}_{1}\circ\tau_\alpha=p_F$ where $\operatorname{pr}_{1}:V_\alpha\times \R^f \to V_\alpha$ is the first projection
\[
\xymatrix{
	F_{|V_\alpha}=p_F^{-1} (V_\alpha)	\ar[rr]^{\tau_{\alpha}}\ar[d]^{p_F} 
		&&	V_\alpha \times \R^f\ar[lld]^{\operatorname{pr}_1}\\
				V_\alpha	&&
}
\]
 and the restriction  $\tau_{\alpha,y}$ of $\tau_\alpha$ to the fibre $F_y:= p_F^{-1}(y)$ is a linear space isomorphism onto $\{y\}\times \R^f$ for all $y\in V_\alpha$.
\item[\textbf{(TFB2)}]
If $\left(  V_{\alpha},\tau_{\alpha}\right)  $ and $\left( V_{\beta},\tau_{\beta} \right) $ are two trivializations such that\\ $V_{\alpha,\beta}:=V_{\alpha}\cap V_{\beta}\neq\emptyset$, the map
\[
\tau_{\alpha,\beta,y}=\tau_{\beta,y} 
\circ\left(  \tau_{\alpha,y}\right) ^{-1}:
\R^f \to \R^f .
\]
is an isomorphism for all $y\in V_{\alpha,\beta}$.
\item[\textbf{(TFB3)}]
For any trivializations  $ \left( V_{\alpha},\tau_{\alpha} \right)  $ and $ \left(
V_{\beta},\tau_{\beta} \right)  $ as previously defined, the transition function $\theta_{\alpha,\beta}$ is characterized by
\[
 \tau_{\beta,y} \circ \left( \tau_{\alpha,y}\right) ^{-1} (y,u)
= \left( y,\theta_{\alpha,\beta}(y,u) \right)  .
\]
and the map $y \mapsto \theta_{\alpha,\beta}(y,.)$ is a smooth map from $V_{\alpha,\beta}$ to $GL(\R^f)$.
\end{enumerate}
Moreover, the transition functions fulfil the cocycle conditions:
\begin{enumerate}
\item[\textbf{(CC1)}]
$\forall y \in V_\alpha \cap V_\beta \cap V_\gamma, \, 
\theta_{\beta,\gamma}(y) \circ \theta_{\alpha,\beta}(y) = \theta_{\alpha,\gamma}(y)$;
\item[\textbf{(CC2)}]
$\forall y \in V_\alpha,\, \theta_{\alpha,\alpha}(y) = \Id_{\R^f}$.
\end{enumerate}
As previously, there exists a countable covering $\{U_\lambda\}_{\lambda \in L}$ of $M$ by open sets $U_\lambda$ of $M$ and associated trivializations on the bundle $E$.

\subsection{Vector bundle homomorphisms}
\label{_VectorBundleHomomorphisms}

A \emph{vector bundle homomorphism} $\Phi : E \to  F$ which covers $\varphi : M \to N$ is a smooth map such that
\begin{enumerate}
\item[\textbf{(LBM1)}]
the following diagram is commutative
\begin{equation}
\label{eq_DiagPhi}
\xymatrix{
           E \ar[rr]^{\Phi}\ar[d]_{p_E} & &F \ar[d]^{p_F}\\
           M \ar[rr]^{\varphi}	&& N\\
}
\end{equation}
\item[\textbf{(LBM2)}]
$\forall x \in M,\,
\Phi_x:=\Phi_{\vert E_x} :  E_x \to F_{\varphi(x)}$ is linear.
\end{enumerate} 

As previously, there exists a covering $\{U_\lambda, \lambda\in L\}$ of $M$ by open set $U_\lambda$ of $M$ and trivialization
$\sigma_\lambda: p_E^{-1}(U_\lambda) \to U_\lambda \times \R^e$ which satisfies the assumption  \textbf{(TFB1)}.\\
Assume that $\varphi(U_\lambda)\cap V_\alpha\not=\emptyset$. If we restrict the diagram (\ref{eq_DiagPhi}) to $U_\lambda$ and $V_\alpha$, the maps
$\Phi_\lambda^\alpha:=\tau_\alpha\circ \Phi\circ \sigma_\lambda^{-1}:U_\lambda\times\R^e\to V_\alpha\times \R^f$ 
can be written
\begin{equation}
\label{eq_LocPhi}
   \Phi_\lambda^\alpha = \varphi_{\vert U_\lambda} \times A_\lambda^\alpha 
\end{equation}
where $A_\lambda^\alpha$ is a smooth map from $U_\lambda^\alpha:=U_\lambda\cap \varphi^{-1}(V_\alpha)$ 
 to $GL(\R^e,\R^f)$.

\subsection{Pullback vector bundles}
\label{__PullBackVectorBundles}

For any smooth function $\varphi : M \to N$, one can define the \emph{pullback vector bundle} $ p_F^!: \varphi^! F \to M$  of $p_F:F\to N$ by 
\[
\varphi^! F 
= \{ (x,v) \in M \times F:\varphi(x)= p_F(v) \},
\]
whose typical fibre is also $\R^f$; we then have a vector bundle homomorphism $ \varphi^!$ which covers $\varphi$ defined by 
\begin{equation}
\label{eq_varphi!}
 \varphi^!(x,v)=(\varphi(x), v)
\end{equation}
where $v \in F_{\varphi(x)}$ and such that the following diagram is commutative\footnote{cf. \cite{AbTo11}.}
\textbf{(LBM!)} 
\hspace*{3cm}
\xymatrix{
           \varphi^! F  \ar[rr]^{\varphi^! }\ar[d]_{p^!_F} & &F \ar[d]^{p_F}\\
           M \ar[rr]^{\varphi}	&& N\\
}
\\
and one has 
$\forall x \in M,\,
 \varphi^!_x :  \left( \varphi^! F \right) _x \to F_{\varphi(x)}$ is an isomorphism
(in particular  \textbf{(LBM2)} is satisfied).\\

According to context of $\S$\ref{_TrivializationsAndTransitionFunctionsOfALinearBundle} and  $\S$\ref{_VectorBundleHomomorphisms}, we have  local trivializations  $\sigma_\lambda$ of  $  p_E^{-1}(U_\lambda)$ with $\lambda \in L$, and  $\tau_\alpha $ of  $p_F^{-1}(V_\alpha)$, with  $\alpha\in A$ respectively.  We denote by $U_\lambda^\alpha=U_\lambda\cap\varphi^{-1}(V_\alpha)$. 
Note that $\{ U_\lambda^\alpha \}_ {(\lambda,\alpha)\in L\times A} $ is a covering of $M$ such that $E_{\vert U_\lambda^\alpha}$ is trivializable with trivialization $\sigma_\lambda$ in restriction 
to $U_\lambda^\alpha$ and $\varphi^!F_{\vert U_\lambda^\alpha}$ is also trivializable with trivialization $(\tau^!)_\lambda^\alpha(x,v)=(x, \tau_{\alpha,  \varphi(x)}(v))$ in restriction to  
$U_\lambda^\alpha$.  
Then according to  notations given in $\S$\ref{_VectorBundleHomomorphisms},  
the expression of $\varphi^!$ in the local charts whose domains are  
$ \left( U_\lambda^\alpha \times \mathbb{R}^f \right) $ and $ \left( V_\alpha \times \mathbb{R}^f \right) $ is
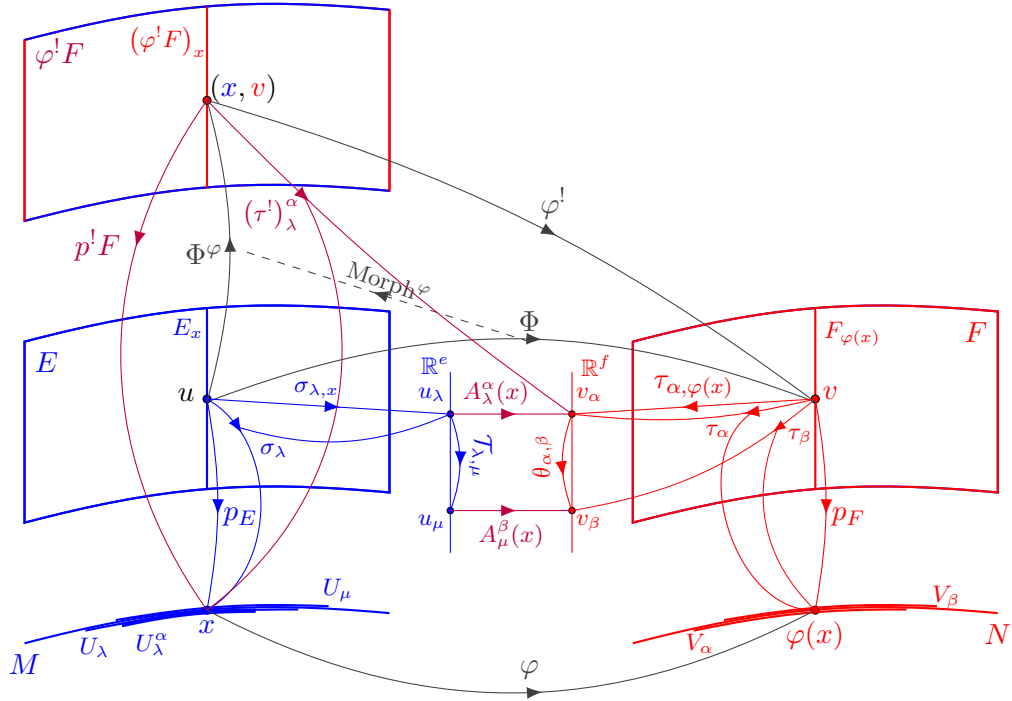
\[
(\varphi^!)_{\lambda}^\alpha
=
\varphi_{\vert U_\lambda^\alpha}
\times
\Id_{\R^f}.
\]
\begin{center}
\begin{tikzpicture}[x=0.8cm,y=0.8cm]
\tikzstyle
simple20=[postaction={decorate,decoration={markings, mark=at position .2 with {\arrow[scale=1.5,>=latex]{>}}}}]
\tikzstyle
simple30=[postaction={decorate,decoration={markings, mark=at position .3 with {\arrow[scale=1.5,>=latex]{>}}}}]
\tikzstyle simple55=[postaction={decorate,decoration={markings, mark=at position .55 with {\arrow[scale=1.5,>=latex]{>}}}}]
\tikzstyle simple65=[postaction={decorate,decoration={markings, mark=at position .65 with {\arrow[scale=1.5,>=latex]{>}}}}]
\draw[blue,thick,opacity=0.8](0,0) to[bend left=10] ++(6,0.5);
\draw[blue,thick,opacity=0.8](0,0) node[below]{$M$};
\draw[fill=blue](3,0.55) circle(1.5pt);
\draw[blue,thick,opacity=0.8] (3,0.55) node[below]{$x$};
\draw[blue,thick](1,0.21) to[bend left=6] ++ (3.5,0.35);
\draw[blue](1.15,-0.07)node{\footnotesize{$U_\lambda $}};
\draw[blue,thick](1.5,0.385) to[bend left=6] ++ (3.5,0.23);
\draw[blue](5.15,0.88)node{\footnotesize{$U_\mu$}};
\draw[blue,ultra thick](1.6,0.3) to[bend left=6] ++ (2.2,0.24);
\draw[blue](2.1,-0.02)node{\footnotesize{$U_\lambda^\alpha $}};
\draw [thick, draw=blue] 
(0,2) to[bend left=10] ++(6,0.5) --++(0,3) to[bend left=-10] ++(-6,-0.5) -- cycle;
\draw [thick, draw=blue] 
(0,2) to[bend left=10] ++(6,0.5) --++(0,3) to[bend left=-10] ++(-6,-0.5) -- cycle;
\draw[blue](0,5) node[below right]{$E$};
\draw[blue,thick](3,2.55)--++(0,3);
\draw[blue](3.1,5.55)node[below left]{\footnotesize{$E_x$}};
\draw[fill=blue,opacity=0.8](3,4.05) circle(1.5pt);
\draw (2.35,4.4) node[below right]{$u$};
\draw[simple55,blue](3,4.05) to[bend right=-10] ++ (0,-3.5);
\draw[blue](3.55,2.1)node{$p_E$};
\draw[red,thick,opacity=0.8](10,0) to[bend left=10] ++(6,0.5);
\draw[red,thick,opacity=0.8](16,0.5) node[below]{$N$};
\draw[fill=red](13,0.55) circle(1.5pt);
\draw[red,thick,opacity=0.8] (13,0.55) node[below]{$\varphi(x)$};
\draw[red,thick](11,0.21) to[bend left=6] ++ (3.5,0.35);
\draw[red](11.1,0)node{\footnotesize{$V_\alpha$}};
\draw[red,thick](11.5,0.385) to[bend left=6] ++ (3.5,0.23);
\draw[red](15.15,0.77)node{\footnotesize{$V_\beta$}};
\draw[simple55,darkgray](3,0.55) to[bend left=-28] ++ (10,0);
\draw[darkgray](8.3,-0.75) node[above]{$\varphi$};
\draw [thick, draw=blue] 
(10,2) to[bend left=10] ++(6,0.5) --++(0,3) to[bend left=-10] ++(-6,-0.5) -- cycle;
\draw [thick, draw=red] 
(10,2) to[bend left=10] ++(6,0.5) --++(0,3) to[bend left=-10] ++(-6,-0.5) -- cycle;
\draw[red](16,5.5) node[below left]{$F$};
\draw[red,thick](13,2.55)--++(0,3);
\draw[red](14.25,5.5)node[below left]{\footnotesize{$F_{\varphi(x)}$}};
\draw[fill=red,opacity=0.8](13,4.05) circle(1.5pt);
\draw[red](13.25,4.15)node{$v$};
\draw[simple55,red](13,4.05) to[bend left=10] ++ (0,-3.5);
\draw[red](13.1,2.1)node[right]{$p_F$};
\draw[simple55,darkgray](3,4) to[bend left=21] ++ (10,0);
\draw[darkgray](8.3,5) node[above]{$\Phi$};
\draw[blue](7,1.5)--(7,4.5);
\draw[blue](6.75,4.6)node{\footnotesize{$\mathbb{R}^e$}};
\draw[fill=blue,opacity=0.8](7,3.8) circle(1.2pt);
\draw[blue](6.3,4.4) node[below right]{\footnotesize{$u_\lambda$}};
\draw[fill=blue,opacity=0.8](7,2.2) circle(1.2pt);
\draw[blue](6.3,2.3) node[below right]{\footnotesize{$u_\mu$}};
\draw[simple55,blue](3,4.05) -- (7,3.8);
\draw[blue](4.8,4.2)node{\footnotesize{$\sigma_{\lambda,x}$}};
\draw[simple20,blue] (3,4.05) to[bend left=60] (3,0.55);
\draw[blue] (3.45,3.6) to[bend right=20] (7,3.8);
\draw[blue] (4.1,3.2) node{\footnotesize{$\sigma_\lambda$}}; 
\draw[simple55,blue](7,3.8) to[bend left=20] (7,2.2);
\node[blue,rotate=-90] at (7.5,3.1) {\footnotesize{$\mathcal{T}_{\lambda,\mu}$}}; 
\draw[red](9,1.5)--(9,4.5);
\draw[red](9.4,4.6)node{\footnotesize{$\mathbb{R}^f$}};
\draw[fill=red,opacity=0.8](9,3.8) circle(1.2pt);
\draw[red](8.9,4.4) node[below right]{\footnotesize{$v_\alpha$}};
\draw[fill=red,opacity=0.8](9,2.2) circle(1.2pt);
\draw[red](8.9,2.3) node[below right]{\footnotesize{$v_\beta$}};
\draw[simple55,red](13,4.05) -- (9,3.8);
\draw[red](11,4.25) node{$\tau_{\alpha,\varphi(x)}$};
\draw[red](12,3.85) to[bend right=75] (13,0.55);
\draw[simple30,red](13,4.05) to[bend left=10] (9,3.8);
\draw[red](11.4,3.52)node{\footnotesize{$\tau_\alpha$}};
\draw[red,very thin](12.4,3.6) to[bend right=35] (13,0.55);
\draw[simple20,red,very thin](13,4.05) to[bend left=15] (9,2.2);
\draw[red](12.75,3.42)node{\footnotesize{$\tau_\beta$}};
\draw[simple65,red,thin](9,3.8) to[bend right=20] (9,2.2);
\node[red,rotate=100] at (8.5,3.1) {\footnotesize{$\theta_{\alpha,\beta}$}};
\draw[simple55,purple](7,3.8) -- (9,3.8);
\draw[purple](7.75,4.15) node{\footnotesize{$A_\lambda^\alpha(x)$}};
\draw[simple55,purple](7,2.2) -- (9,2.2);
\draw[purple](8,1.8) node{\footnotesize{$A_\mu^\beta(x)$}};
\draw [thick, draw=purple] 
(0,7) to[bend left=10] ++(6,0.5) --++(0,3) to[bend left=-10] ++(-6,-0.5) -- cycle;
\draw[blue,thick,opacity=0.8](0,7) to[bend left=10] ++(6,0.5);
\draw[blue,thick,opacity=0.8](0,10) to[bend left=10] ++(6,0.5);
\draw[red,thick](0,7)--++(0,3);
\draw[red,thick](6,7.5)--++(0,3);
\draw[purple](0,10.2) node[below right]{$\varphi^! F$};
\draw[red,thick](3,7.55)--++(0,3);
\draw[red](3.15,10.4)node[below left]{\footnotesize{$\left( \varphi^!F \right) _x$}};
\draw[fill=red,opacity=0.8](3,9) circle(1.5pt);
\draw(3.6,9.2) node{$(\textcolor{blue}{x},\textcolor{red}{v})$};
\draw[simple55,darkgray](3,9) to[bend left=10] (13,4);
\draw[darkgray](8.7,7.3)node{$\varphi^!$};
\draw[simple55,darkgray](3,4) to[bend right=15](3,9);
\draw[darkgray](2.95,6.45)node{$\Phi^\varphi$};
\draw[dashed,simple55,darkgray](8.3,5) -- (3.6,6.5);
\draw[darkgray](5.98,5.95) node{\rotatebox[origin=c]{-19}{\footnotesize{$\Morph^\varphi$}}};
\draw[simple30,purple](3,9) to[bend right=35](3,0.55);
\draw[purple](1.2,6.6)node{$p^!F$};
\draw[simple30,purple](3,9) to[bend right=5] (9,3.8);
\draw[purple](4.59,7.41) to[bend left=40] (3,0.55);
\draw[purple](4.05,7.1)node{\footnotesize{$\left( \tau^! \right) _\lambda^\alpha$}};
\end{tikzpicture}
\end{center}

\section{The convenient manifold $C^\infty(M,N)$}
\label{_TheConvenientManifoldCintfyMN}

Even if the linear space $C^\infty  \left( \mathtt{U},\R^n \right) $, where $\mathtt{U}$ is an open set of $\mathbb{R}^m$, can be endowed with a structure of Fr\'echet space (cf. \cite{Bou81}, EVT III, 7. Exemples, for $n=1$), it is not possible to get, in general, such a structure\footnote{If the manifold $M$ is compact, then $C^\infty(M,N)$ is metrizable (cf. \cite{KrMi97}, 42., Proposition 42.3).} on the space $C^\infty(M,N)$ of smooth maps from the finite dimensional manifold $M$ into the finite dimensional manifold $N$.\\

In fact, following \cite{KrMi97},  Theorem~42.1, the space $C^\infty(M,N)$ can be endowed with a structure of convenient manifold\footnote{Remark that $C^\infty(M,N)$ can be endowed with other topologies:
\begin{enumerate}
\item[$\bullet$]
the \emph{strong $C^\infty$ topology} introduced by Mather (cf. \cite{Ill03}) is important because various important sets, such as submersions, embeddings,... are open in $C^\infty(M,N)$  (cf. \cite{Com14}, Theorem 3.);
\item[$\bullet$]
the Whitney $C^\infty$ topology $W = \displaystyle\bigcup_{k=0}^\infty W_k$ where a basis of neighboroods of the Whitney $C^k$ topology is 
$\{ \varphi \in C^\infty(M,N): \left( j^k \varphi \right)(M) \subset U \}$ where $U$ is an open set of the $k$-jets bundle $J^k(M,N)$. For this topology, $C^\infty(M,N)$ is a Baire space. 
\end{enumerate}
}   
 modelled on the space $C^\infty_c (M \leftarrow \varphi^! TN)$ of compact support sections of the pullback bundles $\varphi^! TN\to M$ along $\varphi:M \to N$ over $M$, endowed with an adequate bornological topology.\\


More precisely,  since $M$ and $N$ are finite dimensional manifolds, the set of smooth sections $C^\infty (M \leftarrow \varphi^! TN)$ has a structure of convenient space  (cf. \cite{KrMi97}, 30.1)\footnote{According to \cite{KrMi97}, Lemma 30.3, if, moreover, $M$ is a first countable topological space, then  $C^\infty (M \leftarrow \varphi^! TN)$ has in fact structure of Fr\'echet space.}.  If $K$ is any compact of $M$, the set of smooth sections  $C^\infty_K (M \leftarrow \varphi^! TN)$ with support in $K$ is a closed subspace of $C^\infty (M \leftarrow \varphi^! TN)$ and so is a convenient subspace of $C^\infty (M \leftarrow \varphi^! TN)$. 
The space $C^\infty_c (M \leftarrow \varphi^! TN)$ of smooth sections with compact support can be endowed with a convenient structure, strict and regular inductive limit of the spaces $C^\infty_{K_n} (M \leftarrow \varphi^! TN)$ where $\left( K_n \right) _{n \in \N}$ is  an exhaustive countable base of compact sets in $M$ (cf. \cite{KrMi97}, Lemma 30.4). In particular, a subset 
 in $C^\infty_c (M \leftarrow \varphi^! TN)$  is $c^\infty$ open if and only if its trace on each such space is open.\\

Now, let us describe this topology on $C^\infty(M,N)$. Consider the exponential $\exp: \mathbf{U} \to N$ associated to a given Riemannian metric on $N$, where $\mathbf{U}$ is an open neighbourhood of the zero section in $TN$ such that $ \left( \pi_N,\exp \right) : \mathbf{U} \to N \times N$ is a smooth diffeomorphism onto an open neighbourhood $\mathbf{V}$ of the diagonal.\\

We then consider the pullback vector bundle
\[
\xymatrix{
           \varphi^! TN  \ar[rr]^{\varphi^!}\ar[d]_{ \pi_N^!} & &TN \ar[d]^{\pi_N}\\
           M \ar[rr]^{\varphi}	&& N\\
}
\]
For any $\psi \in C^\infty(M,N)$, $\psi \sim \varphi$ means that $\varphi$ and $\psi$ agree off some compact subset in $M$.\\
Consider the set
\[
\mathbf{V}_\varphi 
= \{ \psi \in C^\infty(M,N):
 \forall x \in M, 
(\varphi(x),\psi(x)) \in \mathbf{V}, \psi \sim \varphi \}
\]
and the mapping
\[
\mathbf{v}_\varphi : \mathbf{V}_\varphi \to 
C^\infty_c(M \leftarrow \varphi^! TN)
\]
defined, for any $\psi \in C^\infty(M,N)$ and any $x \in M$, by
\[
\mathbf{v}_\varphi(\psi)(x)
= \left( x,\exp^{-1}_{\varphi(x)}(\psi(x)) \right).
\]

\begin{remark}
\label{R_HomotopyVarphiPsi}
If $\psi$ belongs to $\mathbf{V}_\varphi$, let $\widehat{\psi}\in C^\infty_c(M \leftarrow \varphi^! TN)$ such that $(\mathbf{v}_\varphi^{-1}(\widehat{\psi}(x))=\psi(x)$.  For any $t\in [0,1]$, consider the time dependent section 
$\widehat{\psi}_t$ defined by 
$\widehat{\psi}_t(x)=t\widehat{\psi}(x)$,  then $\psi_t(x)=\mathbf{v}_\varphi^{-1} \left( \widehat{\psi}_t(x) \right) $ defines a homotopy from $\varphi$ to $\psi$. Therefore,  
 by  homotopy invariance vector bundles Theorem  (cf. for instance \cite{Hus93}, Theorem 4.7, p.30), 
$\varphi^! TN$ and $\psi^! TN$ are isomorphic and so are  $C^\infty_c(M \leftarrow \varphi^! TN)$ and $C^\infty_c(M \leftarrow \psi^! TN)$.\\
\end{remark}
We then endow $C^\infty(M,N)$ with a structure of convenient manifold \textit{via} the atlas 
$\{ \left( \mathbf{V}_\varphi,\mathbf{v}_\varphi \right) \} _{\varphi \in C^\infty(M,N)}$ 
where the chart change mappings are given for any 
$s \in \mathbf{v}_\psi \left( \mathbf{V}_\varphi \cap \mathbf{V}_\psi \right) $ by
\begin{eqnarray}
\label{eq_ChartChangingCinftyMN}
\left( \mathbf{v}_\varphi \circ \mathbf{v}_\psi^{-1} \right) (s)
=
\left( 
\operatorname{Id}_M, (\pi_N,\exp)^{-1}
\circ 
(\varphi,\exp \circ (\pi_N^\ast \psi) \circ s)
\right).
\end{eqnarray}

\section{Structure on $\Morph(E,F)$}

\subsection{Associated morphisms to a morphism in  $\Morph(E,F)$ }
\label{__AssociatedMorphisms}

In this subsection, the smooth map $\varphi : M \to N$ is fixed.

\begin{proposition}
\label{P_MorphismPhivarphi} 
Let $\Phi : E \to F$  be a vector bundle morphism which covers $\varphi$.
The map
\[
\begin{array}{cccc}
\Phi^\varphi:	&E		&\to	& \varphi^! F	\\
		&(x,u)	&\mapsto& 
		 \left( x,  (\varphi^!_x)^{-1}\circ \Phi_x(u) \right) 
\end{array}
\]
 is a well defined vector bundle morphism over $\operatorname{Id}_M$.\\
Conversely, if 
\[
\Psi: E	\to	\varphi^!F
\]
 is a vector bundle morphism which covers $\Id_M$, the  map
\begin{equation}
\label{eq_ConversePhivarphi}
\Phi = \varphi^!\circ \Psi
\end{equation}
is a vector  bundle morphism which covers $\varphi$. Moreover, we have $\Phi^\varphi=\Psi$.
\end{proposition}

\begin{proof}  
First, note that from (\ref{eq_varphi!}), the linear map $(\varphi^!_x)$ is invertible and so $\Phi^\varphi(x,u)$ is well defined.  We must show that $\Phi^\varphi$ is a smooth map.  Since it is a local problem, we restrict ourselves to the context of $\S$\ref{_VectorBundleHomomorphisms} and $\S$\ref{__PullBackVectorBundles}; choose any $x\in M$  such that $x$ belongs to  $  U_\lambda $ and $
\varphi(x)$ belongs to some  $ V_\alpha$; we then have $U_\lambda^\alpha\not=\emptyset$. Then the bundles $E_{\vert U_\lambda^\alpha}$, $F_{\vert V_\alpha}$ and $\varphi^!F_{\vert U_\lambda^\alpha}$ are trivializable. Then, from the definition of $\Phi^\varphi$, and according to (\ref{eq_LocPhi}), we have:  
\begin{equation}
\label{eq_PhiVarphiLambdaAlpha}
(\Phi^\varphi)_\lambda^\alpha:=\tau_\alpha\circ \Phi^\varphi\circ \left((\tau^!)_\lambda^\alpha\right)^{-1}
=
\left( \Id_{U_\lambda^\alpha},\tau_{\alpha,\varphi(x)}\circ \Phi\circ (\sigma_{\lambda,x} ^{-1})_{\vert U_\lambda^\alpha\times \R^e} 
\right)
\end{equation}
with $\sigma_{\lambda,x}=\operatorname{pr}_2^e \circ \sigma_\lambda$ and 
$\tau_{\alpha,\varphi(x)}=\operatorname{pr}_2^f \circ \tau_\alpha$ where $\operatorname{pr}_2^e$ (resp. $\operatorname{pr}_2^f$) is the projection on the typical type fibre $\R^e$ (resp. $\R^f$).
So we get 
$ \left( \Phi^\varphi \right) _\lambda^\alpha
=
\left( \Id_{U_\lambda^\alpha},A_\lambda^\alpha \right) $ and then $(\Phi^\varphi)_\lambda^\alpha$ is a smooth map, for any such open sets $U_\lambda^\alpha$ and $V_\alpha$. This implies that  $\Phi^\varphi$ is a smooth map. Finally,  from its definition, $\Phi^\varphi$ satisfies the condition \textbf{(LBM!)}.\\

For the converse,  the relation (\ref{eq_ConversePhivarphi}) defines a vector bundle morphism  as a composition of such morphisms and it is clear that we have $\Phi^\varphi=\Psi$.
\end{proof}

From Proposition \ref{P_MorphismPhivarphi}, we have  a bijection
\begin{equation}
\label{eq_varphiindex}
\begin{array}{cccc}
\Morph^\varphi:	&\Morph_\varphi (E,F) 	& \to	&\Morph_{\Id_M}  \left( E,\varphi^! F \right)	\\		
			& \Phi	& \mapsto 	& \Phi ^\varphi
\end{array}
\end{equation}
where 
\begin{enumerate}
\item[$\bullet$]
$\Morph_{\Id_M}  \left( E,\varphi^! F \right) $ is the set of vector bundle morphisms: $E \to \varphi^! F$ above $\Id_M$ 
\item[$\bullet$]
$\Morph_\varphi (E,F)$ is the set of vector bundle morphisms: $E \to F$ above $\varphi$.
\end{enumerate}

It is easy to see  that 
$L(E,\varphi^! F):= \displaystyle\bigcup_{x \in M} L  \left( E_x,F_{\varphi(x)} \right)$  has a  structure of vector bundle  over $M$ whose typical fibre is $L  \left( \R^e,\R^f \right)$ and structural group $GL\left(\R^e\right)\times GL\left(\R^f\right)$. 
So $\Morph_{\Id_M} \left( E,\varphi^! F \right)$ is the space of smooth sections of  $L(E,\varphi^! F)\to M$ and, according to \cite{KrMi97}, 30, this set is   denoted  $C^\infty\left(M \leftarrow  L\left(E,\varphi^! F\right))\right)$  and has a structure of Fr\'echet space from Proposition 30.1.

\begin{remark}
\label{R_IsomorphismMorphIdM} ${}$
\begin{enumerate}
\item
Note that $\Morph_\varphi (E,F)$ and  $\Morph_{\Id_M}  \left( E,\varphi^! F \right)$ have a real vector space structure and it is easy to see that  $\Morph^\varphi$ is linear.  Since  $\Morph^\varphi$ is a bijection, then we can provide $\Morph_\varphi (E,F)$ with a structure of convenient  space and so $\Morph^\varphi$ is  an isomorphism.  
\item 
Assume that $\psi$ is homotopic to $\varphi$. Then by  Homotopy invariance vector bundles Theorem  (cf. \cite{Hus93}, Theorem 4.7, p.30), there exists  a bundle isomorphism $H_{\varphi,\psi}$ from 
$\varphi^! F$ to $\psi^! F$,  and so we get  an  isomorphism    from $\Morph_{\Id_M}  \left( E,\varphi^! F \right) $  to $\Morph_{\Id_M}  \left( E,\psi^! F \right) $ given by 
$\Morph^\psi = H_{\varphi,\psi} \circ \Morph^\varphi$. 
\end{enumerate}
\end{remark} 

\subsection{Structure of convenient manifold on $\Morph(E,F)$}
\label{_StructureOfConvenientManifoldOnMorphEF}

Consider
\[
\begin{array}{cccc}
\pi:	& \Morph(E,F)	&\to	&C^\infty(M,N)\\
		& (\varphi,\Phi)&\mapsto&\varphi
\end{array}
\]
where one has
\[
\Morph(E,F) 
= \displaystyle\bigcup_{\varphi \in C^\infty(M,N)}
\Morph_{\Id_M} \left(E,\varphi^! F \right) 
\] 

We will provide a convenient manifold structure  on $\Morph(E,F)$ such that $ \pi: \Morph(E,F)	\to	C^\infty(M,N)$ is a convenient vector bundle.
\begin{theorem} 
\label{T_ConvenientMorph(EF)}
There exists a convenient manifold structure on $\Morph(E,F)$  modelled on $C^\infty_c(M \leftarrow \varphi^!TN)\times C^\infty\left(M \leftarrow  L \left(E,\varphi^! F\right) \right)$.\\
Moreover, 
\[
\pi:  \Morph(E,F)\to C^\infty(M,N)
\]
is a convenient vector bundle  with typical fibre $\Morph_{\Id_M} \left(E,\varphi^! F \right)$.
\end{theorem}

\begin{remark}
\label{R_CasParticulier}
According to \cite{KrMi97}, 42.4, Theorem~\ref{T_ConvenientMorph(EF)} is also true if $N$ has a structure of convenient manifold which can be provided with a local addition. On the other hand, if $M$ is compact, then $C^\infty (M,N)$ can be endowed with a structure of Fr\'echet manifold and $\Morph_\varphi (E,F)$ is a Fr\'echet space; so  $\pi:  \Morph(E,F)\to C^\infty(M,N)$ is a Fr\'echet vector bundle. 
\end{remark}

\begin{proof}
We consider a chart $ \left( \mathbf{V}_\varphi,\mathbf{v}_\varphi \right) $ around $\varphi$ in $C^\infty(M,N)$  and according to Remark \ref{R_IsomorphismMorphIdM},  consider the  map
\[
 \begin{array}{cccc}

\mathbf{T}^{\varphi}:	

	&	\pi^{-1}(\mathbf{V}_\varphi)

	&	\to 	

	&	\mathbf{V}_\varphi\times \Morph_{\Id_M}  

	\left( E,\varphi^! F \right)	\\

	& (\psi,\Psi)

	& \mapsto 

	&  \big( \psi,  \left( H_{\varphi,\psi} \right) ^{-1} \circ \Morph^{\psi}(\Psi)  \big)  

\end{array}
\]
We then get the following commutative diagram: 
\[
\xymatrix{
	\Morph(E,F)_{\mathbf{V}_\varphi}
	=\pi^{-1} (\mathbf{V}_\varphi)	\ar[rr]^{\mathbf{T}^{\varphi}}\ar[d]^{\pi} 
		&&	\mathbf{V}_\varphi\times \Morph_{\Id_M}  \left( E,\varphi^! F \right)  \ar[lld]^{\operatorname{pr}_1}\\
		\mathbf{V}_\varphi	&&
}
\] 
From the definition of $ \mathbf{T}^{\varphi}$,  its  restriction  to  $\pi^{-1}(\psi)=\Morph_\psi (E,F)$ is 
 $ \left( H_{\varphi,\psi} \right) ^{-1} \circ \Morph^\psi(\Psi)$.  According to Remark \ref{R_IsomorphismMorphIdM},   $H_{\varphi,\psi}$ is a diffeomorphism from  $\varphi^! F$ to $\psi^! F$   and so  $ H_{\varphi,\psi}^{-1} \circ \Morph^\psi(\Psi)  $ 
is a diffeomorphism  from 
$\Morph_\psi (E,F)$  to  $\Morph_{\Id_M}  \left( E,\varphi^! F \right) $.\\
It follows that the map $ \mathbf{T}^{\varphi}$ is bijective and so we can provide $\Morph(E,F)_{\mathbf{V}_\varphi}$ with the convenient manifold structure modelled on $C^\infty_c(M \leftarrow \varphi^!TN)\times C^\infty\left(M\leftarrow L(\R^e, \R^f)\right)$ for which $\mathbf{T}^{\varphi}$ is a diffeomorphism. Note that, since the restriction  $ \mathbf{T}^{\varphi}$ to  $\pi^{-1}(\psi)$ is 
$\left( H_{\varphi,\psi} \right) ^{-1} \circ \Morph^\psi(\Psi)$, the convenient structure of manifold on each  $\pi^{-1}(\psi)$  defined in Remark  \ref{R_IsomorphismMorphIdM}, Assertion 1 is a  supplemented convenient submanifold of $\Morph(E,F)_{\mathbf{V}_\varphi}$. This implies that  $ \mathbf{T}^{\varphi}$ is a trivialization  of  $\pi^{-1}(\mathbf{V}_\varphi)$ and the restriction of $\pi$ to this convenient manifold is a surjective submersion.  \\
Consider a chart $\mathbf{V}_\psi$ in  $C^\infty(M,N)$ such that $\mathbf{V}_\psi\cap\mathbf{V}_\varphi\not=\emptyset$.\\

The transition functions  
{\footnotesize
\[
 \mathbf{T}^{\varphi}\circ ( \mathbf{T}^{\psi})^{-1}: \; \mathbf{T}^\psi \Big( \left( \mathbf{V}_\varphi \cap \mathbf{V}_\psi \right) 
\times \Morph_{\Id_M} \left( E,\psi^! F \right)
\Big)
\to
\mathbf{T}^\varphi \Big( \left( \mathbf{V}_\varphi \cap \mathbf{V}_\psi \right) 
\times \Morph_{\Id_M} \left( E,\varphi^! F \right)
\Big)  
\]
}
  are the maps given by
\begin{equation}
\label{eq_TransitionTV}
 \mathbf{T}^{\varphi}\circ ( \mathbf{T}^{\psi})^{-1}
 =
 \left( \mathbf{v}_\varphi \circ \mathbf{v}_\psi^{-1}, 
 (\Morph^\varphi)(\Phi)^{-1}\circ(\Morph^\psi(\Psi))  \right).
\end{equation}
where $\mathbf{v}_\varphi \circ \mathbf{v}_\psi^{-1}$ corresponds to the change of charts  (\ref{eq_ChartChangingCinftyMN}).\\ 
As we have already seen, 
\[
\Pi_\varphi:\Morph_{\Id_M}  \left( E,\varphi^! F \right) \to M:= C^\infty \left( M \leftarrow L(\R^e, \R^f) \right)
\] 
has a convenient  vector space structure. 
Then if 
 $ \left\{ \left( \mathbf{V}_\varphi,\mathbf{v} \right) \right\}
 _{\varphi\in C^\infty (M,N)}$ is an atlas on $ C^\infty (M,N)$, it follows that  $\left\{ \left(\pi^{-1}((\mathbf{V}_\varphi)),(\mathbf{v}_\varphi, \mathbf{T}_\phi) \right) \right\}$ is an atlas  which defines  a convenient structure manifold on $\Morph(E,F)$  modelled on  $C^\infty_c(M \leftarrow \varphi^!TN)\times  \Morph_{\Id_M}  \left( E,\varphi^! F \right)$  and $\pi:  \Morph(E,F)\to C^\infty(M,N)$ is also a convenient  vector bundle with typical fibre $\Morph_{\Id_M} \left(E,\varphi^! F \right)$. 
\end{proof}

\begin{remark}
\label{Morph(E,F)NotSubmanifold}
Since $E$ and $F$ are finite dimensional fibre bundles over $M$ and $N$ respectively, $E$ and $F$ are in particular second countable finite dimensional manifolds, the  set $C^\infty (E,F)$ can be provided with a convenient  structure of manifold modelled on $C^\infty_c(E \leftarrow \varphi^\ast TF)$, as proved in $\S$\ref{_TheConvenientManifoldCintfyMN} for the set $C^\infty (M,N)$ of smooth maps between second countable finite-dimensional manifolds.\\

As $\Morph(E,F)$ is a subset of  $C^\infty (E,F)$ which is provided with its own structure of convenient manifold, it is natural to ask us if, in this context, \\

\emph{Is $\Morph(E,F)$ a weak convenient submanifold of $C^\infty(E,F)$?} \\

Recall  that if $X_1$ and $X_2$ are convenient manifolds respectively modelled on the convenient spaces $\mathbb{X}_1$ and $\mathbb{X}_2$, 
a smooth map $f:X_1 \to X_2$ is called
a weak immersion if, for any $x\in X_1$, there exist a chart $(U_1,\phi_1)$ around $x\in X_1$ and a chart $(U_2,\phi_2)$ around $f(x)\in X_2$ such that
\begin{equation}
\label{Eq_weakimmersion}
\phi_2\circ f\circ \phi_1^{-1}=\iota_{\phi_1(U_1)}
\end{equation}
where $\iota :\mathbb{X}_1\to \mathbb{X}_2$ is a smooth 
  linear injective map.  \\

Then $\iota$ is continuous for the convenient topology, i.e., for any $c^\infty$ open set $\mathbb{U}_2\subset \mathbb{X}_2$, $ \iota^{-1}(\mathbb{U}_2)$ is a $c^\infty$-open set in $\mathbb{X}_1$.  This implies that the inclusion of $X_1$ into $X_2$ is a smooth injective and continuous map between $X_1$ and $X_2$. Now, if, moreover, $Z$ is a weak submanifold of $X_1$, since the composition of smooth convenient maps is a smooth map and the composition of continuous maps is continuous, the  inclusion of $Z$ into $X_2$ is also an injective convenient smooth continuous map from $Z$ to $X_2$. We will apply the argument for the inclusion of the fibre $\pi^{-1}(\varphi)$ to $C^\infty(E,F)$.\\

\emph{For an answer to the above question, argue by contradiction: assume that the answer is yes.} \\

Let $\Phi$ be a bundle morphism from $E$ to $F$, there exists a chart $(\mathbf{W}_\Phi,\mathbf{w}_\Phi)$  around $\Phi$, and, for any $\Psi\in \mathbf{W}_\Phi$, there exists a compact subset $K_\Psi$ in $E$ such that $\Phi=\Psi$ on $E\setminus K_\Psi$.   In particular, if $\Phi$ (resp. $\Psi$) covers $\varphi$ (resp. $\psi$) then $\varphi=\psi$ on $p_E(E\setminus K_\Psi)\subset M$.  Now we have seen in the proof of Theorem \ref{_StructureOfConvenientManifoldOnMorphEF} that $\pi^{-1}(\varphi)$ is a submanifold of $\Morph(E,F)$. So, accoding to the previous general argument, under our assumption, $\Morph_\varphi(E,F)\cap \mathbf{W}_\Phi$, for  some chart domain $\mathbf{W}_\Phi$ in $C^\infty(E,F)$, is open in $\Morph_\varphi(E,F)$. Therefore, there exist vector  bundle morphisms  $\Phi$ and $\Psi $ which cover $\varphi$, such that $\Phi \not= \Psi$ and which belongs to $\Morph_\varphi(E,F)\cap \mathbf{W}_\Phi$. 
But, there exists a compact set $K_\Psi$ in $E$ such that  $\Phi=\Psi$ on $E\setminus K_\Psi$. But  since  $\Psi\not=\Phi$, there must exist $(x,u)\in K_\Psi$ such that $\Phi(x,u)\not=\Psi(x,u)$. However, there exists $n\in \mathbb{N}\setminus\{0\}$ such that $(x,nu)\not\in K_\Psi$ and so $\Psi(x,nu)=\Phi(x,nu)$.  
Since $\Psi$ is linear in the second variable, one has $n \Psi(x,u)=n \Phi(x,u)$, so we get $\Psi(x,u)= \Phi(x,u)$ because $n \neq 0$.  
Now, this result is also true for any $x\in \pi(K_\Psi)$ and this implies that $\Psi=\Phi$, which means that  $\Morph_\varphi(E,F)\cap \mathbf{W}_\Phi=\{\Phi\}$ which is not an open set of  $\Morph_\varphi(E,F)$.  By the above general argument, we get a contradiction. Therefore, $\Morph(E,F)$ is {\bf not} a weak convenient submanifold of $C^\infty(E,F)$.
\end{remark}

\end{document}